\documentclass[10pt]{amsart}
\usepackage[utf8]{inputenc}
\usepackage{graphicx}
\usepackage[centertags]{amsmath}
\usepackage{amsfonts}
\usepackage{amssymb}
\usepackage{amsthm}
\usepackage{newlfont}
\usepackage{amsaddr}
\newtheorem{thm}{Theorem}[section]
\newtheorem{cor}[thm]{Corollary}

\newtheorem{prop}[thm]{Proposition}
\theoremstyle{definition}

\theoremstyle{remark}

\newcommand{\id}{\textrm{id}}
\usepackage{lipsum}
\begin{document}
%%%%%%%%%%%%%%%%%%%%%%%%%%%%%%%%%
\title[On a Riemannian manifold with a circulant structure]{On a Riemannian manifold with a circulant structure whose third power is the identity}
\bigskip
\author{Iva Dokuzova}

\date{}
%\date{21. 10. 2010}%
%\dedicatory{math}%
%\commby{}%
% ----------------------------------------------------------------
\address{Department of
Algebra and Geometry\\ Faculty of Mathematics and Informatics \\ University of Plovdiv Paisii Hilendarski\\ 24 Tzar Asen, 4000 Plovdiv, Bulgaria}
\email{dokuzova@uni-plovdiv.bg}

%%% IF THERE ARE MORE THAN TWO AUTHORS WRITE
%%% \newcommand{\AuthorNames}{First Author et al.}
%%%
%%% ENTER MSC, KEYWORDS, RECEIVED, EDITOR, THANKS FOR FINANCIAL SUPPORT FOR RESEARCH
\thanks{This work is partially supported by project FP17-FMI-008 of the Scientific Research Fund, Paisii Hilendarski University of Plovdiv, Bulgaria.}
\begin{abstract}
%%%%%%%%%%%%%%%%%%%%%%%%%%%%%%%%%%%

It is studied a $3$-dimensional Riemannian manifold equipped with a tensor structure of type $(1,1)$, whose third power is the identity.
This structure has a circulant matrix with respect to some basis, i.e. the structure is circulant. On such a manifold a fundamental tensor by the metric and by the covariant derivative of the circulant structure is defined. An important characteristic identity for this tensor is obtained.
It is established that the image of the fundamental tensor with respect to the usual conformal transformation satisfies the same identity.
A Lie group as a manifold of the considered type is constructed and some of its geometrical characteristics are found.
\end{abstract}

\maketitle
% ----------------------------------------------------------------
\section{Introduction}
In differential geometry of the Riemannian manifolds with additional structures,  the covariant derivative of the corresponding structure plays an important role. In this connection, for example, the classifications in \cite{GH}, \cite{8} and \cite{S-G} are made. One of the basic classes in these classifications is the largest class, which is invariant under the conformal transformations of the Riemannian metric.

 In \cite{2}, \cite{4} and \cite{AE}, problems of differential geometry of a $3$-dimensional Riemannian manifold $(M, g, Q)$ with a tensor structure $Q$, whose matrix in some basis is circulant, are considered. This structure satisfies $Q^{3}=\id$, $Q\neq\id$, and it is compatible with the metric $g$, so that an isometry is induced in any tangent space on $(M, g, Q)$.

 In the present work we continue studying such a manifold $(M, g, Q)$. In Section~\ref{2}, we give some necessary facts about $(M, g, Q)$. In Section~\ref{3}, we define the fundamental tensor $F$ by the metric $g$ and by the covariant derivative of $Q$. We obtain the important characteristic identity \eqref{c1} for $F$. We establish that the image of the fundamental tensor with respect to the usual conformal transformation satisfies the same identity, i.e. the conformal manifold $(M, \overline{g}, Q)$ belongs to the same class. In Section~\ref{4}, we get some curvature properties of $(M, g, Q)$. In Section~\ref{5},
 we construct a manifold $(G, g, Q)$ of the type of $(M, g, Q)$, where $G$ is a Lie group. In Section~\ref{6}, we consider a subgroup $G'$ of $G$, for which the manifold $(G', g, Q)$ is of invariant sectional curvatures under $Q$. We study some particular cases.

\section{Preliminaries}\label{2}
\thispagestyle{empty}
We consider the manifold $(M, g, Q)$, introduced in \cite{2}, i.e. $M$ is a $3$-dimensional Riemannian manifold equipped with an additional tensor structure $Q$ of type $(1,1)$, which satisfies
   \begin{equation}\label{Q3}
    Q^{3}=\id,\ Q\neq\id,
\end{equation}
and $Q$ has a circulant matrix with respect to some basis, as follows:
\begin{equation}\label{Q}
    (Q_{j}^{s})=\begin{pmatrix}
      0 & 1 & 0 \\
      0 & 0 & 1 \\
      1 & 0 & 0 \\
    \end{pmatrix}.
\end{equation}
The metric $g$ and the structure $Q$ satisfy
\begin{equation}\label{g3}
 g(Qx, Qy)=g(x,y),\quad x, y\in \mathfrak{X}(M).
\end{equation}
Necessary, the matrix of $g$ with respect to the same basis has the form
\begin{equation}\label{metricg}
    (g_{ij})=\begin{pmatrix}
      A & B & B \\
      B & A & B \\
      B & B & A \\
    \end{pmatrix},
\end{equation}
where $A$ and $B$ are smooth functions of an arbitrary point $p(x^{1}, x^{2},
x^{3})$ on $M$.
It is supposed that $A>B>0$ in order $g$ to be positive definite.

Moreover, in \cite{2} it is defined another metric $\tilde{g}$ associated to $g$ by $Q$, as follows:
\begin{equation}\label{metricf}
 \tilde{g}(x,y)=g(x, Qy)+g(Qx, y).\end{equation}
Here and anywhere in this work, $x, y, z, u$ will stand for arbitrary elements of the algebra on the smooth vector fields on $M$ or vectors in the tangent space $T_{p}M$, $p\in M$. The Einstein summation convention is used, the range of the summation indices being always $\{1, 2, 3\}$.

\section{The fundamental tensor $F$ on $(M, g, Q)$}\label{3}

Let $\nabla$ be the Levi-Civita connection of $g$ and $\{e_{i}\}$ be an arbitrary basis of $T_{p}M$. We consider the tensor $F$ of type $(0,3)$ and the Lee forms $\theta$ and $\theta^{*}$, defined by
\begin{equation}\label{F}
  F(x,y,z)=(\nabla_{x}\tilde{g})(y, z),\quad \theta(x)=g^{ij}F(e_{i},e_{j},x),\quad
  \theta^{*}(x)=g^{ij}F(e_{i},Qe_{j},x).
\end{equation}
Obviously, the tensor $F$ has the property
\begin{equation}\label{prop-F}
  F(x,z,y)=F(x,y,z).
\end{equation}

\begin{thm}
For the tensor $F$ it is valid the identity
\begin{equation}\label{c1}
 F(x,y,z)=\frac{1}{3}\big\{g(x,y)\theta(z)+g(x,z)\theta(y)+\tilde{g}(x,y)\theta^{*}(z)+\tilde{g}(x,z)\theta^{*}(y)\big\}.
\end{equation}
\end{thm}
\begin{proof}
The components of the geometric quantities of $(M, g, Q)$, given in the formulas  \eqref{g-obr}, \eqref{gamma} \eqref{f2.1} and \eqref{f-obr}, are obtained in \cite{2} and \cite{4}.
The inverse matrix of $(g_{ij})$ is
\begin{equation}\label{g-obr}
    (g^{ij})=\frac{1}{D}\begin{pmatrix}
      A+B & -B & -B \\
      -B & A+B & -B \\
      -B & -B & A+B \\
    \end{pmatrix},
\end{equation}
where $D=(A-B)(A+2B)$.
The Christoffel symbols of $g$ are as follows:
\begin{equation}\label{gamma}
\begin{split}
\Gamma_{ij}^{k}&=\frac{1}{2D}\big((A+B)(-B_{k}+B_{i}+B_{j})-B(A_{i}+A_{j})\big),\\
\Gamma_{ii}^{i}&=\frac{1}{2D}\big((A+B)A_{i}-B(4B_{i}-A_{j}-A_{k})\big),\\
\Gamma_{ii}^{k}&=\frac{1}{2D}\big((A+B)(2B_{i}-A_{k})-B(2B_{i}-A_{j}+A_{i})\big),\\
\Gamma_{ij}^{i}&=\frac{1}{2D}\big((A+B)A_{j}-B(-B_{k}+B_{i}+B_{j}+A_{i})\big),\\
\end{split}
\end{equation}
where $i\neq j$, $j\neq k$, $i\neq k$ and
  $A_{i}=\frac{\partial A}{\partial x^{i}},$ $B_{i}=\frac{\partial B}{\partial x^{i}}$.
The  matrix  of the associated metric $\tilde{g}$, determined by \eqref{metricf}, is
of the type:
\begin{equation}\label{f2.1}
(\tilde{g}_{ij})=\begin{pmatrix}
      2B & A+B & A+B \\
      A+B & 2B & A+B \\
      A+B & A+B & 2B \\
    \end{pmatrix}.
\end{equation}
The inverse matrix of $(\tilde{g}_{ij})$ has the form
\begin{equation}\label{f-obr}
    (\tilde{g}^{ij})=\frac{1}{2D}\begin{pmatrix}
      -A-3B & A+B & A+B \\
      A+B & -A-3B & A+B \\
      A+B & A+B & -A-3B \\
    \end{pmatrix}.
\end{equation}

Now, we calculate the components of $F$, $\theta$ and $\theta^{*}$, defined by
\eqref{F}.

Using the well-known identities for a Riemannian metric:
\begin{equation}\label{defF}
\nabla_{k}\tilde{g}_{ij}=\partial_{k}\tilde{g}_{ij}-\Gamma_{ki}^{s}\tilde{g}_{sj}-\Gamma_{kj}^{s}\tilde{g}_{si}\ ,
\end{equation}
and due to \eqref{prop-F}, \eqref{f2.1} and \eqref{gamma}, we find the following components $F_{ijk}=F(e_{i}, e_{j}, e_{k})$ of $F$:
\begin{align}\label{nablaf}\nonumber
&F_{111}=-2B_{1}+A_{2}+A_{3},\quad
F_{211}=F_{311}=-B_{1}+B_{2}+B_{3}-A_{1},\\\nonumber
&F_{112}=F_{221}=\frac{1}{2}(A_{3}-B_{1}-B_{2}+B_{3}),\quad
F_{312}=B_{3}-\frac{1}{2}(A_{1}+A_{2}),\\\nonumber
&F_{222}=-2B_{2}+A_{1}+A_{3},\quad
F_{322}=F_{122}=B_{1}-B_{2}+B_{3}-A_{2},\\
&F_{113}=F_{331}=\frac{1}{2}(A_{2}-B_{1}+B_{2}-B_{3}),\quad
F_{213}=B_{2}-\frac{1}{2}(A_{1}+A_{3}),\\\nonumber
&F_{123}=B_{1}-\frac{1}{2}(A_{2}+A_{3}),\quad
F_{223}=F_{332}=\frac{1}{2}(A_{1}+B_{1}-B_{2}-B_{3}),\\\nonumber & F_{133}=F_{233}=B_{1}+B_{2}-B_{3}-A_{3},\quad
F_{333}=-2B_{3}+A_{1}+A_{2}.\\\nonumber
\end{align}

For the components of $\theta$ and $\theta^{*},$ from \eqref{Q}, \eqref{F}, \eqref{g-obr} and \eqref{nablaf}, we have
\begin{align}\label{alfa}\nonumber
\theta_{1}&=\frac{3}{2D}\big((A+B)(A_{2}+A_{3})+2B(A_{1}-B_{2}-B_{3})-2AB_{1}\big),\\
\theta_{2}&=\frac{3}{2D}\big((A+B)(A_{1}+A_{3})+2B(A_{2}-B_{1}-B_{3})-2AB_{2}\big),\\\nonumber
\theta_{3}&=\frac{3}{2D}\big((A+B)(A_{1}+A_{2})+2B(A_{3}-B_{1}-B_{2})-2AB_{3}\big),\\\nonumber
\end{align}
\begin{align}\label{tild-alfa}\nonumber
\theta^{*}_{1}&=-\frac{3}{2D}\big((A-2B)B_{1}-A(B_{2}+B_{3}-A_{1})+B(A_{2}+A_{3})\big),\\
\theta^{*}_{2}&=-\frac{3}{2D}\big((A-2B)B_{2}-A(B_{1}+B_{3}-A_{2})+B(A_{1}+A_{3})\big),\\\nonumber
\theta^{*}_{3}&=-\frac{3}{2D}\big((A-2B)B_{3}-A(B_{1}+B_{2}-A_{3})+B(A_{1}+A_{2})\big).\\\nonumber
\end{align}
From \eqref{metricg}, \eqref{f2.1}, \eqref{nablaf}, \eqref{alfa} and \eqref{tild-alfa} it follows
 \begin{equation}\label{w1}
 \begin{split}
  F_{kij}=\frac{1}{3}\big(g_{kj}\theta_{i}+g_{ki}\theta_{j}+\tilde{g}_{kj}\theta^{*}_{i}+\tilde{g}_{ki}\theta^{*}_{j}\big),\\
  \end{split}
\end{equation}
which is equivalent to \eqref{c1}.
\end{proof}

\begin{thm}\label{connF-lineF}
Under the conformal transformation
 \begin{equation}\label{conf}\overline{g}(x, y) = \alpha g(x, y),\end{equation}
where $\alpha$ is a smooth positive function, the tensor $F$ is transformed into the tensor
\begin{equation}\label{overlineF}
 \overline{F}(x,y,z)=\frac{1}{3}\big\{\overline{g}(x,y)\overline{\theta}(z)+\overline{g}(x,z)\overline{\theta}(y)+\overline{\widetilde{g}}(x,y)\overline{\theta}^{*}(z)+\overline{\widetilde{g}}(x,z)\overline{\theta}^{*}(y)\big\}
\end{equation}
with  $\overline{\theta}=\theta+\dfrac{3}{2\alpha}\mathrm{d}\alpha\circ (Q+ Q^{2})$ and $\overline{\theta}^{*}=\theta^{*}-\dfrac{3}{2\alpha}\mathrm{d}\alpha.$
\end{thm}
\begin{proof}
We denote the following products
 \begin{equation}\label{fi-and-s}
 \widetilde{g}_{ij}g^{is}=\Phi_{j}^{s}\ ,\quad g_{ij}\widetilde{g}^{is}=\frac{1}{2}S_{j}^{s}\ ,
 \end{equation}
where $\Phi_{j}^{s}$ and $\frac{1}{2}S_{j}^{s}$ are mutually inverse matrices.
Because of \eqref{metricg}, \eqref{g-obr},  \eqref{f2.1} and \eqref{f-obr}  we have
\begin{equation}\label{S-Phi}
(\Phi_{j}^{s})=\begin{pmatrix}
      0 & 1 & 1 \\
      1 & 0 & 1 \\
      1 & 1 & 0 \\
    \end{pmatrix},\qquad
 (S_{j}^{s})=\begin{pmatrix}
      -1 & 1 & 1 \\
      1 & -1 & 1 \\
      1 & 1 & -1 \\
    \end{pmatrix}.
    \end{equation}
Having in mind \eqref{Q} and \eqref{S-Phi}, we get $\Phi=Q+Q^{2}$.

From \eqref{alfa}, \eqref{tild-alfa} and the second matrix of \eqref{S-Phi}, we get
 \begin{equation}\label{theta-s*}
 \begin{split}
  \theta^{*}_{i}=-\frac{1}{2}S_{i}^{s}\theta_{s}\ .
  \end{split}
\end{equation}

According to the transformation \eqref{conf}, the components of the tensor $\overline{F}$ are $\overline{F}_{ijk}=\overline{\nabla}_{i}\overline{\widetilde{g}}_{jk}$, where $\overline{\nabla}$ is the Levi-Civita connection of $\overline{g}$.

 Bearing in mind \eqref{metricf} and \eqref{conf}, we have that $\overline{\widetilde{g}}=\alpha \widetilde{g}$. Then \begin{equation}\label{barF}\overline{\nabla}\ \overline{\widetilde{g}}=\alpha\overline{\nabla}\widetilde{g}+\widetilde{g}\overline{\nabla}\alpha .\end{equation}

From the Christoffel formulas
\begin{equation}\label{cristofel}
 2\Gamma^{k}_{ij}=g^{ks}(\partial_{i}g_{sj}+\partial_{j}g_{si}-\partial_{s}g_{ij}), \ 2\overline{\Gamma}^{k}_{ij}=\overline{g}^{ks}(\partial_{i}\overline{g}_{sj}+\partial_{j}\overline{g}_{si}-\partial_{s}\overline{g}_{ij})
 \end{equation}
 and \eqref{conf} we get
 \begin{equation*}
 \overline{\Gamma}^{k}_{ij}= \Gamma^{k}_{ij}+\frac{1}{2\alpha}(\delta_{j}^{k}\alpha_{i}+\delta_{i}^{k}\alpha_{j}-g_{ij}g^{ks}\alpha_{s}),\quad
 \alpha_{s}=\frac{\partial\alpha}{\partial x^{s}}\ .
 \end{equation*}
 Then, using \eqref{defF} for $\overline{\nabla}\widetilde{g}$, we obtain
 \begin{equation}\label{F-barF}
 \begin{split}
  \overline{\nabla}_{k}\widetilde{g}_{ji}=\nabla_{k}\widetilde{g}_{ji} &-\frac{1}{2\alpha}(\widetilde{g}_{ji}\alpha_{k}+\widetilde{g}_{ik}\alpha_{j}-g_{kj}\Phi_{i}^{s}\alpha_{s})\\&-\frac{1}{2\alpha}(\widetilde{g}_{kj}\alpha_{i}+\widetilde{g}_{ij}\alpha_{k}-g_{ik}\Phi_{j}^{s}\alpha_{s}).
\end{split}
\end{equation}

Substituting \eqref{w1} and \eqref{F-barF} into \eqref{barF}, we get
 \begin{equation}\label{w1-2}
 \begin{split}
  \overline{F}_{kij}=\frac{1}{3}\big(\overline{g}_{kj}\overline{\theta}_{i}+\overline{g}_{ki}\overline{\theta}_{j}+\overline{\widetilde{g}}_{kj}\overline{\theta}^{*}_{i}+\overline{\widetilde{g}}_{ki}\overline{\theta}^{*}_{j}\big),\\
   \overline{\theta}_{i}=\theta_{i}+\frac{3}{2\alpha}\Phi_{i}^{s}\alpha_{s}\ ,\quad \overline{\theta}^{*}_{i}=-\frac{1}{2}S_{i}^{s}\overline{\theta}_{s}\ .
\end{split}
\end{equation}
Then, due to \eqref{Q}, \eqref{S-Phi} and \eqref{theta-s*},
 for $(M, \overline{g}, Q)$ it is valid the identity \eqref{overlineF}.
\end{proof}

Note. According to Theorem~\ref{connF-lineF}, we can say that $(M, g, Q)$ and $(M, \overline{g}, Q)$ belong to classes of the same type, defined by the equality \eqref{c1} for the corresponding metric.

Immediately, from \eqref{F} and \eqref{overlineF}, we have the following
\begin{cor}
 If $F=0$ holds, then it is valid
 \begin{equation}\label{W0-glob}
  \overline{F}(x, y,z)=\frac{1}{2}\big\{\overline{g}(x, y)\alpha(\Phi z)+\overline{g}(x, z)\alpha(\Phi y)-\overline{\widetilde{g}}(x, y)\alpha(z)-\overline{\widetilde{g}}(x, z)\alpha(y)\big\}.
\end{equation}
\end{cor}
Next, we obtain
\begin{cor}
 If $F=0$ holds, then $\overline{F}$ vanishes if and only if $\alpha$ is a constant function.
\end{cor}
\begin{proof} In a local form \eqref{W0-glob} is
\begin{equation}\label{W0}
  \overline{F}_{kij}=\frac{1}{2}\big(\overline{g}_{kj}\Phi_{i}^{s}\alpha_{s}+\overline{g}_{ki}\Phi_{j}^{s}\alpha_{s}-\overline{\widetilde{g}}_{kj}\alpha_{i}-\overline{\widetilde{g}}_{ki}\alpha_{j}\big).
\end{equation}
Let $\overline{F}=0$ be valid. Then, from \eqref{conf} and \eqref{W0}, it follows $$g_{kj}\Phi_{i}^{s}\alpha_{s}+g_{ki}\Phi_{j}^{s}\alpha_{s}-\widetilde{g}_{kj}\alpha_{i}-\widetilde{g}_{ki}\alpha_{j}=0.$$
 Contracting by $g^{kj}$ in the latter equality, and using \eqref{fi-and-s} and \eqref{S-Phi}, we get
    $\Phi_{i}^{s}\alpha_{s}=0$
which implies $\alpha_{1}=\alpha_{2}=\alpha_{3}=0$, i.e. $\alpha$ is a constant.

Vice versa. If $\alpha$ is a constant, then its partial derivatives are $\alpha_{1}=\alpha_{2}=\alpha_{3}=0$, and having in mind \eqref{W0}, we get $\overline{F}=0$.
\end{proof}

\section{Some curvature properties of $(M, g, Q)$}\label{4}
It is well-known, that the curvature tensor $R$ of $\nabla$ is defined by
\begin{equation}\label{R}
R(x, y)z=\nabla_{x}\nabla_{y}z-\nabla_{y}\nabla_{x}z-\nabla_{[x,y]}z.
\end{equation}
Also, we consider the tensor of type $(0, 4)$ associated with $R$, defined as follows:
\begin{equation}\label{R2}
    R(x, y, z, u)=g(R(x, y)z,u).
\end{equation}
The Ricci tensor $\rho$ and the scalar curvature $\tau$ with respect to $g$ are as usually
\begin{equation}\label{def-rho}
    \rho(y,z)=g^{ij}R(e_{i}, y, z, e_{j}),\quad
    \tau=g^{ij}\rho(e_{i}, e_{j}).
\end{equation}

Let $\tilde{\Gamma}$ and $\tilde{\nabla}$ be the Christoffel symbols and the Levi-Civita connection of $\tilde{g}$, respectively. Let $\tilde{R}$
be the curvature tensor of $\tilde{\nabla}$. The Ricci tensor $\tilde{\rho}$ and the scalar curvature $\tilde{\tau}$ with respect to $\tilde{g}$ are
\begin{equation}\label{def-rho2}
    \tilde{\rho}(y,z)=\tilde{g}^{ij}\tilde{R}(e_{i}, y, z, e_{j}),\
    \tilde{\tau}=\tilde{g}^{ij}\tilde{\rho}(e_{i}, e_{j}).
\end{equation}

 From \eqref{defF}, using the Christoffel formulas \eqref{cristofel} for $\Gamma$ and $\tilde{\Gamma}$, we obtain the following relation
\begin{equation}\label{gamma3}
 \tilde{\Gamma}^{k}_{ij}=\Gamma^{k}_{ij}+\frac{1}{2}\tilde{g}^{ks}(\nabla_{i}\tilde{g}_{js}+\nabla_{j}\tilde{g}_{is}-\nabla_{s}\tilde{g}_{ij}).
 \end{equation}
 We substitute \eqref{w1} into \eqref{gamma3} and we get
  \begin{equation}\label{gamma4}
 \tilde{\Gamma}^{k}_{ij}=\Gamma^{k}_{ij}+\frac{1}{3}\tilde{g}^{ks}(g_{ij}\theta_{s}+\tilde{g}_{ij}\theta^{*}_{s}).
 \end{equation}

From  \eqref{fi-and-s}, \eqref{S-Phi} and \eqref{theta-s*}, we find
  \begin{equation}\label{tild-g}
    \tilde{g}^{sk}\theta_{s}=-\theta^{*k},\ \tilde{g}_{sk}\theta^{s}=\theta_{k}-2\theta^{*}_{k},\ \tilde{g}_{sk}\theta^{*s}=-\theta_{k},\ \tilde{g}^{sk}\theta^{*}_{s}=-\frac{1}{2}(\theta ^{k}+\theta^{*k}).
\end{equation}
We apply \eqref{tild-g} into \eqref{gamma4} and we obtain that the tensor $\textrm{T}=\widetilde{\Gamma}-\Gamma$ of the affine deformation of $\nabla$ has components
 \begin{equation}\label{torsion}
 T^{k}_{ij}=-\frac{1}{6}\big(2g_{ij}\theta^{*k}+\tilde{g}_{ij}(\theta^{k}+\theta^{*k})\big).
 \end{equation}

It is well-known the relation
$ \widetilde{R}^{k}_{ijs} = R^{k}_{ijs} + \nabla_{j}T^{k}_{is}-\nabla_{s}T^{k}_{ij}
+T^{a}_{is}T^{k}_{aj}-T^{a}_{ij}T^{k}_{as}.$
Then, using \eqref{theta-s*}, \eqref{tild-g} and \eqref{torsion}, we
calculate
\begin{equation*}
\begin{split}
     \tilde{R}^{k}_{ijs}& = R^{k}_{ijs} - \frac{1}{3}g_{is}(\nabla_{j}\theta^{*k}-\frac{1}{3}\theta^{*}_{j}\theta^{*k})+\frac{1}{3}g_{ij}(\nabla_{s}\theta^{*k}-\frac{1}{3}\theta^{*}_{s}\theta^{*k})\\
 &-\frac{1}{6}\tilde{g}_{is}(\nabla_{j}\theta^{*k}+\nabla_{j}\theta^{k}-\frac{1}{3}\theta_{j}\theta^{*k}-\frac{1}{3}\theta^{*}_{j}\theta^{*k})\\&+\frac{1}{6}\tilde{g}_{ij}(\nabla_{s}\theta^{*k}+\nabla_{s}\theta^{k}-\frac{1}{3}\theta_{s}\theta^{*k}-\frac{1}{3}\theta^{*}_{s}\theta^{*k}).
\end{split}
\end{equation*}
By contracting $k=s$ in the latter equality, and having in mind  \eqref{w1}, \eqref{fi-and-s}, \eqref{theta-s*}, \eqref{def-rho}, \eqref{def-rho2} and \eqref{tild-g}, we get
\begin{equation}\label{tilde-S}
\begin{split}
     \tilde{\rho}_{ij} = \rho_{ij}+\frac{1}{3}g_{ij}\nabla_{s}\theta^{*s}+\frac{1}{6}\tilde{g}_{ij}(\nabla_{s}\theta^{s}+\nabla_{s}\theta^{*s}).
\end{split}
\end{equation}

Let us denote
\begin{equation}\label{def-rho*}
    \tau^{*}=\tilde{g}^{ij}\rho(e_{i}, e_{j}),\quad \tilde{\tau}^{*}=g^{ij}\tilde{\rho}(e_{i}, e_{j}).
\end{equation}
Due to \eqref{def-rho}, \eqref{def-rho2}, \eqref{tilde-S} and \eqref{def-rho*} we obtain
\begin{equation*}
     \tilde{\rho}_{ij} = \rho_{ij}+\frac{1}{3}(\tilde{\tau}^{*}-\tau)g_{ij}+\frac{1}{6}(2\tilde{\tau}-2\tau^{*}+\tilde{\tau}^{*}-\tau)\tilde{g}_{ij},
\end{equation*}
Therefore we establish the following
\begin{thm}\label{connR-R}
 For the Ricci tensors  $\rho$ and $\tilde{\rho}$ and for the scalar curvatures $\tau$, $\tau^{*}$, $\tilde{\tau}$ and $\tilde{\tau}^{*}$ the following relation is valid:
 \begin{equation}\label{con-AE}
     \tilde{\rho}(x,y) = \rho(x,y)+\frac{1}{3}(\tilde{\tau}^{*}-\tau)g(x,y)+\frac{1}{6}(2\tilde{\tau}-2\tau^{*}+\tilde{\tau}^{*}-\tau)\tilde{g}(x,y).
\end{equation}
\end{thm}

In \cite{AE}, a Riemannian manifold $(M, g, Q)$ is called
almost Einstein if the metrics $g$ and $\tilde{g}$ satisfy
\begin{equation}\label{AE}\rho(x, y) = \beta g(x, y) + \gamma \tilde{g}(x, y),\end{equation} where $\beta$ and $\gamma$ are smooth functions on $M$.

\begin{cor}\label{th2.4}
   If the Levi-Civita connection $\tilde{\nabla}$ of $\tilde{g}$ is a locally flat connection, then $(M, g, Q)$ is an almost Einstein manifold and the Ricci tensor $\rho$ has the form
 \begin{equation}\label{rho-AE}
\rho(x, y) = \frac{\tau}{3}g(x, y)+\frac{2\tau^{*}+\tau}{6}\tilde{g}(x,y).
\end{equation}
\begin{proof} If $\tilde{\nabla}$ is a locally flat connection, then we have $\widetilde{R}=0$. From \eqref{def-rho2} and \eqref{def-rho*} it follows $\widetilde{\rho}=0$ and $\widetilde{\tau}=\widetilde{\tau}^{*}=0$. Then \eqref{con-AE} implies \eqref{rho-AE} and according to \eqref{AE} we have that $(M, g, Q)$ is an almost Einstein manifold.
\end{proof}
\end{cor}
Note. Examples of manifolds $(M, g, Q)$ satisfying \eqref{rho-AE} are considered in \cite{AE}.

In an analogous way we prove
\begin{cor}\label{th2.5}
     If the Levi-Civita connection $\nabla$ of $g$ is a locally flat connection, then $(M, \tilde{g}, Q)$ is an almost Einstein manifold and the Ricci tensor $\tilde{\rho}$ has the form
 \begin{equation*}
\tilde{\rho}(x, y) = \frac{\tilde{\tau}^{*}}{3}g(x, y)+\frac{\tilde{\tau}^{*}+2\tilde{\tau}}{6}\tilde{g}(x,y).
\end{equation*}
\end{cor}

%%% ----------------------------------------------------------------
\section{Lie groups as manifolds of the considered type}\label{5}

Let $G$ be a $3$-dimensional real connected Lie group and $\mathfrak{g}$ be its Lie algebra with a basis $\{x_{1}, x_{2},x_{3}\}$ of left invariant vector fields. We introduce a circulant structure $Q$ and a Riemannian metric $g$ as follows:
\begin{equation}\label{lie}
  Qx_{1}=x_{2},\ Qx_{2}=x_{3},\ Qx_{3}=x_{1},
\end{equation}
  \begin{equation}\label{g}
 g(x_{i}, x_{j})= \left\{ \begin{array}{ll}
                        0, & i\neq j \hbox{;} \\
                        1, & i=j \hbox{.}
                      \end{array}
                    \right.
  \end{equation}
 Obviously, for the manifold $(G, g, Q)$ the equalities \eqref{Q3} and \eqref{g3} are valid, i.e. $(G, g, Q)$ is a Riemannian manifold of the same type as $(M, g, Q)$.

For the associated metric $\tilde{g}$, due to \eqref{metricf}, we get
\begin{equation}\label{f}
 \tilde{g}(x_{i}, x_{j})= \left\{ \begin{array}{ll}
                        1, & i\neq j \hbox{;} \\
                        0, & i=j \hbox{.}
                      \end{array}
                    \right.
  \end{equation}

The corresponding Lie algebra $\mathfrak{g}$ is determined as follows:
\begin{equation}\label{skobki}
  [x_{i}, x_{j}]=C_{ij}^{k}x_{k},
\end{equation}
where $C_{ij}^{k}=-C_{ji}^{k}.$
The Jacobi identity implies  \cite{HM-DM}
\begin{equation}\label{cij-o}
\begin{split}
C_{23}^{1}(C_{31}^{3}-C_{12}^{2})&=C_{23}^{2}C_{21}^{1}+C_{23}^{3}C_{31}^{1},\\ C_{12}^{3}(C_{31}^{1}-C_{23}^{2})&=C_{23}^{3}C_{21}^{2}+C_{31}^{3}C_{12}^{1},\\ C_{13}^{2}(C_{23}^{3}-C_{12}^{1})&=C_{23}^{2}C_{13}^{3}+C_{12}^{2}C_{31}^{1}.
\end{split}
\end{equation}

Let us denote
\begin{equation}\label{C}
\begin{split}
\lambda_{1}=C_{12}^{1},\  \lambda_{2}=C_{12}^{2},\  \lambda_{3}=C_{12}^{3},\\ \mu_{1}=C_{13}^{1},\  \mu_{2}=C_{13}^{2},\  \mu_{3}=C_{13}^{3},\\  \nu_{1}=C_{23}^{1}, \ \nu_{2}=C_{23}^{2},\ \nu_{3}=C_{23}^{3}.
\end{split}
\end{equation}
Then \eqref{skobki}, applying \eqref{C}, takes the form
\begin{equation}\label{skobki-o}
\begin{split}
  [x_{1}, x_{2}]&=\lambda_{1}x_{1}+\lambda_{2}x_{2}+\lambda_{3}x_{3},\\
  [x_{2}, x_{3}]&=\nu_{1}x_{1}+\nu_{2}x_{2}+\nu_{3}x_{3},\\
  [x_{1}, x_{3}]&=\mu_{1}x_{1}+\mu_{2}x_{2}+\mu_{3}x_{3}.\\
  \end{split}
\end{equation}
According to \eqref{cij-o} and \eqref{C}, we obtain the following conditions:
\begin{equation}\label{jakobi}
\begin{split}
\nu_{1}(\mu_{3}-\lambda_{2})&=\nu_{2}\lambda_{1}+\mu_{1}\nu_{3},\\ \lambda_{3}(\mu_{1}-\nu_{2})&=\nu_{3}\lambda_{2}+\mu_{3}\lambda_{1},\\ \mu_{2}(\nu_{3}-\lambda_{1})&=\nu_{2}\mu_{3}+\lambda_{2}\mu_{1}.
\end{split}
\end{equation}
The well-known Koszul formula implies
\begin{equation*}
    2g(\nabla_{x_{i}}x_{j}, x_{k})=g([x_{i}, x_{j}],x_{k})+g([x_{k}, x_{i}],x_{j})+g([x_{k}, x_{j}],x_{i}),
\end{equation*}
and using \eqref{g} and \eqref{skobki-o}, we obtain
\begin{align}\label{nabla}\nonumber
    \nabla_{x_{1}}x_{1}&=-\lambda_{1}x_{2}-\mu_{1}x_{3},\
    \nabla_{x_{1}}x_{2}=\lambda_{1}x_{1}+\frac{1}{2}(\lambda_{3}-\mu_{2}-\nu_{1})x_{3},\\\nonumber
    \nabla_{x_{1}}x_{3}&=\mu_{1}x_{1}+\frac{1}{2}(\mu_{2}-\lambda_{3}+\nu_{1})x_{2},\\\nonumber
    \nabla_{x_{2}}x_{1}&=-\lambda_{2}x_{2}-\frac{1}{2}(\lambda_{3}+\nu_{1}+\mu_{2})x_{3},\\
    \nabla_{x_{2}}x_{2}&=\lambda_{2}x_{1}-\nu_{2}x_{3},\
    \nabla_{x_{2}}x_{3}=\frac{1}{2}(\nu_{1}+\mu_{2}+\lambda_{3})x_{1}+\nu_{2}x_{2},\\\nonumber
    \nabla_{x_{3}}x_{1}&=\frac{1}{2}(\nu_{1}-\mu_{2}-\lambda_{3})x_{2}-\mu_{3}x_{3},\\\nonumber
    \nabla_{x_{3}}x_{2}&=\frac{1}{2}(-\nu_{1}+\mu_{2}+\lambda_{3})x_{1}-\nu_{3}x_{3},\
    \nabla_{x_{3}}x_{3}=\mu_{3}x_{1}+\nu_{3}x_{2}.\\\nonumber
\end{align}
Further, bearing in mind \eqref{F}, \eqref{g} and \eqref{nabla}, for the components of $F$, $\theta$ and $\theta^{*}$ we get
\begin{align}\label{F1-o}\nonumber
  F_{111} & =2\lambda_{1}+2\mu_{1},\quad
  F_{112}  =\frac{1}{2}(2\mu_{1}+\mu_{2}-\lambda_{3}+\nu_{1}), \\\nonumber
  F_{113} & =\frac{1}{2}(2\lambda_{1}+\lambda_{3}-\mu_{2}-\nu_{1}), \quad
  F_{122}  =-2\lambda_{1}+\nu_{1}+\mu_{2}-\lambda_{3}, \\\nonumber
  F_{123} & =-\lambda_{1}-\mu_{1}, \quad
  F_{221}  =\frac{1}{2}(\lambda_{3}+\nu_{1}+\mu_{2}+2\nu_{2}), \\
  F_{222} & =-2\lambda_{2}+2\nu_{2}, \quad
  F_{223}  =-\frac{1}{2}(2\lambda_{2}+\nu_{1}+\mu_{2}+\lambda_{3}), \\\nonumber
   F_{211} & =2\lambda_{2}+\nu_{1}+\mu_{2}+\lambda_{3},\quad
  F_{213} =\lambda_{2}-\nu_{2}, \\\nonumber
  F_{133} &=-2\mu_{1}-\nu_{1}-\mu_{2}+\lambda_{3},\quad F_{233} =-\lambda_{3}-\nu_{1}-2\nu_{2}-\mu_{2}, \\\nonumber
   F_{311} &=\lambda_{3}-\nu_{1}+\mu_{2}+2\mu_{3},\quad F_{312} =\nu_{3}+\mu_{3}, \\\nonumber
    F_{313} &=-\frac{1}{2}(-\lambda_{1}+\nu_{1}-\mu_{2}+2\nu_{3}),\quad F_{322} =-\lambda_{3}+\nu_{1}-\mu_{2}+2\nu_{3}, \\\nonumber
    F_{332}& =-\frac{1}{2}(\lambda_{3}-\nu_{1}+\mu_{2}+2\mu_{3}),\quad F_{333} =-2\mu_{3}-2\nu_{3}, \\\nonumber
\end{align}
\begin{align}\label{theta-o2}\nonumber
\theta_{1}&=2\lambda_{1}+\lambda_{3}+2\mu_{1}+\mu_{2}+\nu_{2}-\nu_{3},\\
\theta_{2}&=-2\lambda_{2}-\lambda_{3}+\mu_{1}-\mu_{3}+\nu_{1}+2\nu_{2},\\\nonumber
\theta_{3}&=\lambda_{1}-\lambda_{2}-\mu_{2}-2\mu_{3}-\nu_{1}-2\nu_{3},\\\nonumber
\end{align}
\begin{align}\label{tild-theta2-o}\nonumber
\theta^{*}_{1}&=-\frac{1}{2}(-\lambda_{1}-3\lambda_{2}-2\lambda_{3}-\mu_{1}-2\mu_{2}-3\mu_{3}+\nu_{2}-\nu_{3}),\\
\theta^{*}_{2}&=-\frac{1}{2}(3\lambda_{1}+\lambda_{2}+2\lambda_{3}+\mu_{1}-\mu_{3}-2\nu_{1}-\nu_{2}-3\nu_{3}),\\\nonumber
\theta^{*}_{3}&=-\frac{1}{2}(\lambda_{1}-\lambda_{2}+3\mu_{1}+2\mu_{2}+\mu_{3}+2\nu_{1}+3\nu_{2}+\nu_{3}).\\\nonumber
\end{align}
\begin{thm}\label{kt2}
The manifold $(G, g, Q)$ has a corresponding Lie algebra determined by
\begin{equation}\label{skobki-o2}
\begin{split}
  [x_{1}, x_{2}]&=\lambda_{1}x_{1}+\lambda_{2}x_{2}+\lambda_{3}x_{3},\\
  [x_{2}, x_{3}]&=\nu_{1}x_{1}+\nu_{2}x_{2}+\nu_{3}x_{3},\\
   [x_{1}, x_{3}]&=(\nu_{2}+\lambda_{3})x_{1}+(\nu_{3}+\lambda_{1})x_{2}+(\nu_{1}+\lambda_{2})x_{3}\\
  \end{split}
\end{equation}
where the structure constants satisfy the following conditions:
\begin{equation}\label{jacobi}
\begin{split}
 2\nu_{1}\lambda_{2}+\nu^{2}_{1}-\lambda_{1}\nu_{2}-\nu_{2}\nu_{3}-\lambda_{3}\nu_{3}=0,\\ 2\nu_{2}\lambda_{3}+\lambda^{2}_{3}-\lambda_{2}\nu_{3}-\lambda_{1}\lambda_{2}-\lambda_{1}\nu_{1}=0,\\  -\lambda^{2}_{1}+\nu^{2}_{3}+\lambda_{2}\lambda_{3}-\nu_{1}\nu_{2}=0.\\
\end{split}
\end{equation}
\end{thm}
\begin{proof}
Since  $(G, g, Q)$ is of the type of the manifold $(M, g, Q)$, the equality \eqref{c1} is valid. Having in mind \eqref{c1}, \eqref{g} and \eqref{f}, we get
\begin{align}\label{F1-lemma}\nonumber
  F_{111} & =\frac{2}{3}\theta_{1},\quad
  F_{112}  =\frac{1}{3}(\theta_{2}+\theta^{*}_{1}),\quad
  F_{113}  =\frac{1}{3}(\theta_{3}+\theta^{*}_{1}),\\\nonumber
  F_{122}  &=F_{322} =\frac{2}{3}\theta^{*}_{2},\quad
  F_{123}  =\frac{1}{3}(\theta^{*}_{3}+\theta^{*}_{2}),\quad
  F_{221}  =\frac{1}{3}(\theta_{1}+\theta^{*}_{2}),\\
  F_{222} & =\frac{2}{3}\theta_{2},\quad
  F_{223}  =\frac{1}{3}(\theta_{3}+\theta^{*}_{2}),\quad
   F_{211}  =F_{311} =\frac{2}{3}\theta^{*}_{1},\\\nonumber
  F_{213} &=\frac{1}{3}(\theta^{*}_{3}+\theta^{*}_{1}),\quad
  F_{133} =F_{233} =\frac{2}{3}\theta^{*}_{3},\quad
   \ F_{312} =\frac{1}{3}(\theta^{*}_{2}+\theta^{*}_{1}),\\\nonumber
    F_{313} &=\frac{1}{3}(\theta_{1}+\theta^{*}_{3}),\quad
    F_{332} =\frac{1}{3}(\theta_{2}+\theta^{*}_{3}),\quad F_{333} =\frac{2}{3}\theta_{3}. \\\nonumber
\end{align}
Then, using \eqref{F1-o}, \eqref{theta-o2} and \eqref{tild-theta2-o}, we find
\begin{equation}\label{mu-theta-o}
  \mu_{1}=\nu_{2}+\lambda_{3}, \quad \mu_{2}=\nu_{3}+\lambda_{1},\quad \mu_{3}=\nu_{1}+\lambda_{2}.
\end{equation}
From \eqref{skobki-o}, \eqref{jakobi} and \eqref{mu-theta-o} we obtain \eqref{skobki-o2} and \eqref{jacobi}.
\end{proof}
 \begin{prop} The components of $R$, $F$, $\theta$ and $\theta^{*}$ with respect to the basis $\{x_{i}\}$ for $(G, g, Q)$ are
  \begin{equation}\label{r1-o}
\begin{split}
    R_{1212}&=-\frac{1}{4}(\lambda_{1}+ \nu_{1}+\nu_{3})^{2}+\frac{3}{4}\lambda^{2}_{3}+\lambda^{2}_{1}+\lambda^{2}_{2}+\nu_{2}\lambda_{3}+\nu^{2}_{2}\\&+\frac{1}{2}\lambda_{1}\lambda_{3}+\frac{1}{2}\nu_{3}\lambda_{3}-\frac{1}{2}\lambda_{3}\nu_{1}, \\
     R_{1313}&=\frac{3}{4}(\lambda_{1}+\nu_{3})^{2}-\frac{1}{4}(\lambda_{3}-\nu_{1})^{2}+(\lambda_{3}+\nu_{2})^{2}+(\lambda_{2}+\nu_{1})^{2}\\&+\frac{1}{2}(\lambda_{1}+\nu_{3})(\nu_{1}+\lambda_{3})-\lambda_{1}\nu_{3}, \\
     R_{2323}&=-\frac{1}{4}\lambda^{2}_{1}+\lambda^{2}_{2}-\frac{1}{4}\lambda^{2}_{3}+\frac{3}{4}\nu_{1}^{2}+\nu^{2}_{2}+\frac{3}{4}\nu_{3}^{2}+\frac{1}{2}\lambda_{1}(-\nu_{3}+\nu_{1}-\lambda_{3})\\&-\frac{1}{2}\lambda_{3}\nu_{3}+\lambda_{2}\nu_{1}-\frac{1}{2}\nu_{1}\lambda_{3}+\frac{1}{2}\nu_{1}\nu_{3}, \\
    R_{1223}&=-\lambda_{2}\lambda_{3}+\lambda_{1}\nu_{1}+\lambda^{2}_{1}+\lambda_{1}\nu_{3}+\lambda_{3}\nu_{3},\\
     R_{1213}&=\lambda_{2}(\lambda_{1}+\nu_{1}+\nu_{3})+\lambda_{3}(\lambda_{1}+\lambda_{2}+\nu_{1}), \\
     R_{1332}&=-\lambda_{1}\nu_{2}-\nu_{2}\nu_{3}-\lambda_{3}\nu_{1}-\nu_{2}\lambda_{3}-\nu_{1}\nu_{3}-\nu_{2}\nu_{1}, \\
         \end{split}
\end{equation}
\begin{align}\label{F2-o}\nonumber
  F_{111} & =2\lambda_{1}+2\lambda_{3}+2\nu_{2},\quad
  F_{112} = F_{221}  =\frac{1}{2}(\lambda_{1}+\lambda_{3}+\nu_{1}+2\nu_{2}+\nu_{3}), \\\nonumber
  F_{113}  &= F_{331} =\frac{1}{2}(\lambda_{1}+\lambda_{3}-\nu_{1}-\nu_{3}),\quad  F_{222}  =-2\lambda_{2}+2\nu_{2}, \\\nonumber
  F_{123}  &=-\lambda_{1}-\lambda_{3}-\nu_{2},\quad
  F_{211} = F_{311} =\lambda_{1}+2\lambda_{2}+\lambda_{3}+\nu_{1}+\nu_{3},\\
  F_{223} & =F_{332}  =-\frac{1}{2}(\lambda_{1}+2\lambda_{2}+\lambda_{3}+\nu_{1}+\nu_{3}), \\\nonumber
  F_{233} &=F_{133} =-\lambda_{1}-\lambda_{3}-\nu_{1}-2\nu_{2}-\nu_{3},\quad F_{312} =\lambda_{2}+\nu_{1}+\nu_{3},\\\nonumber
  F_{133} &=-\lambda_{1}-\lambda_{3}-\nu_{1}-2\nu_{2}-\nu_{3},\ F_{213} =\lambda_{2}-\nu_{2}, \\\nonumber
    F_{122}  &=F_{322} =-\lambda_{1}-\lambda_{3}+\nu_{1}+\nu_{3},\quad
       F_{333} =-2\lambda_{2}-2\nu_{1}-2\nu_{3}, \\\nonumber
\end{align}
\begin{align}\label{theta-o3}\nonumber
\theta_{1}&=3(\lambda_{1}+\lambda_{3}+\nu_{2}),\quad
\theta_{2}=3(\nu_{2}-\lambda_{2}),\quad
\theta_{3}=-3(\lambda_{2}+\nu_{1}+\nu_{3}),\\
\theta^{*}_{1}&=\frac{3}{2}(\lambda_{1}+2\lambda_{2}+\lambda_{3}+\nu_{1}+\nu_{3}),\quad
\theta^{*}_{2}=-\frac{3}{2}(\lambda_{1}+\lambda_{3}-\nu_{1}-\nu_{3}),\\\nonumber
\theta^{*}_{3}&=-\frac{3}{2}(\lambda_{1}+\lambda_{3}+\nu_{1}+2\nu_{2}+\nu_{3}).\\\nonumber
\end{align}
\end{prop}
\begin{proof}
The statement follows from \eqref{R}, \eqref{R2}, \eqref{F1-o}, \eqref{theta-o2}, \eqref{tild-theta2-o} and \eqref{mu-theta-o}.
\end{proof}
\section{Manifolds $(G, g, Q)$ with sectional curvatures invariant under $Q$}\label{6}

If $\{x, y\}$ is a non-degenerate $2$-plane spanned by vectors $x, y$ in the tangent space of the manifold, then its sectional curvature is
\begin{equation}\label{3.3}
    k(x,y)=\frac{R(x, y, x, y)}{g(x, x)g(y, y)-g^{2}(x, y)}\ .
\end{equation}

Let $G'$ be a subgroup of $G$ and $(G', g, Q)$ be a manifold
with sectional curvatures which are invariant under $Q$, i.e. according to \eqref{g} and \eqref{3.3} we have
\begin{equation}\label{V2}
  R(Qx, Qy, Qx, Qy)=R(x, y, x, y).
\end{equation}
Because of \eqref{lie}, the identity \eqref{V2} is equivalent to
\begin{equation}\label{r1=r6}
    R_{1212}=R_{1313}=R_{2323}.
\end{equation}
From \eqref{r1-o} and \eqref{r1=r6}, we find
\begin{equation}\label{R1=R2=R3}
\begin{split}
 \lambda^{2}_{1}+ \lambda^{2}_{3}+ \lambda_{1}\nu_{3}+\lambda_{3}\nu_{3}+ \lambda_{1}\lambda_{3}+ \lambda_{2}\nu_{1}+\lambda_{3}\nu_{1}+ 2\lambda_{3}\nu_{2}=0,\\ \nu^{2}_{1}+ \nu^{2}_{3}+\lambda_{3}\nu_{2}+\lambda_{1}\nu_{3}+ \lambda_{3}\nu_{1}+ \lambda_{1}\nu_{1}+\nu_{3}\nu_{1}+ 2\lambda_{2}\nu_{1}=0.\\
\end{split}
\end{equation}

We will consider the following instances where the system of equations \eqref{jacobi} and \eqref{R1=R2=R3} is executed:

\textbf{Case~(A)}
$\ \lambda_{3}=0,\ \nu_{1}=0,\ \nu_{3}=-\lambda_{1},$

\textbf{Case~(B)}
 $\ \lambda_{3}=-\lambda_{1}-\lambda_{2},\ \nu_{1}=-\lambda_{1}-\lambda_{2},\ \nu_{2}=\lambda_{1},\ \nu_{3}=\lambda_{2},$

 \textbf{Case~(C)}
 $\ \lambda_{3}=-\lambda_{1}-\lambda_{2},\ \nu_{1}=\lambda_{1},\ \nu_{2}=\lambda_{2},\ \nu_{3}=-\lambda_{1}-\lambda_{2}.$

%--------------------------------------------------------------------------------------------------------------------------------------------------------------
Let us consider Case~(A). From \eqref{skobki-o2} and \eqref{jacobi} it follows that the Lie algebra is determined by the commutators:
\begin{equation}\label{skobki-ex1}
\begin{split}
  [x_{1}, x_{2}]&=\lambda_{1}x_{1}+\lambda_{2}x_{2},\\
  [x_{2}, x_{3}]&=\nu_{2}x_{2}-\lambda_{1}x_{3},\\
   [x_{1}, x_{3}]&=\nu_{2}x_{1}+\lambda_{2}x_{3}.\\
  \end{split}
\end{equation}
Then, using \eqref{F2-o} and \eqref{theta-o3}, for the components of  $F$, $\theta$ and $\theta^{*}$ we get
\begin{align}\label{F-ex1}\nonumber
F_{111} & =2\lambda_{1}+2\nu_{2},\quad
  F_{112}  =\nu_{2},\quad
  F_{113}  =\lambda_{1}, \\\nonumber
  F_{122}  &= F_{322} =-2\lambda_{1},\quad
  F_{123}  =-\lambda_{1}-\nu_{2},\quad
  F_{221}  =\nu_{2},\\
  F_{222}  &=-2\lambda_{2}+2\nu_{2},\quad
  F_{223}  =-\lambda_{2},\quad
   F_{211}  =F_{311}=2\lambda_{2},\\\nonumber
  F_{213} &=\lambda_{2}-\nu_{2},\quad
  F_{133} =F_{233} =-2\nu_{2},\quad
   \ F_{312} =\lambda_{2}-\lambda_{1},\\\nonumber
    F_{313} &=\lambda_{1},\quad
    F_{332} =-\lambda_{2},\quad F_{333} =-2\lambda_{2}+2\lambda_{1},\\\nonumber
\end{align}
\begin{equation}\label{theta-ex1}
\begin{array}{lll}
\theta_{1}=3(\lambda_{1}+\nu_{2}),\quad &
\theta_{2}=3(\nu_{2}-\lambda_{2}),\quad &
\theta_{3}=3(\lambda_{1}-\lambda_{2}),\\
\theta^{*}_{1}=3\lambda_{2},\quad &
\theta^{*}_{2}=-3\lambda_{1},\quad &
\theta^{*}_{3}=-3\nu_{2}.
\end{array}
\end{equation}

From \eqref{def-rho} and \eqref{r1-o} we obtain all nonzero components of $R$ and $\rho$ on $(G',g, Q)$:
\begin{equation}\label{R-ex1}
    R_{1212}= R_{2323}=R_{1313}=\lambda^{2}_{1}+\lambda^{2}_{2}+\nu^{2}_{2},
\end{equation}
\begin{equation}\label{rho-ex1}
     \rho_{11}= \rho_{22}=\rho_{33}=-2(\lambda^{2}_{1}+\lambda^{2}_{2}+\nu^{2}_{2}).
\end{equation}
Then the scalar curvature of  $(G',g, Q)$ is
\begin{equation}\label{tau-ex1}
    \tau=-6(\lambda^{2}_{1}+\lambda^{2}_{2}+\nu^{2}_{2}).
\end{equation}
From \eqref{g}, \eqref{rho-ex1} and \eqref{tau-ex1} it follows
$\rho=\dfrac{\tau}{3}g,$
 i.e. $(G', g, Q)$ is an Einstein manifold.

Using \eqref{3.3} and \eqref{r1=r6} we get that the sectional curvatures $k_{ij}$ of the basic $2$-planes $\{x_{i},  x_{j}\}$ are equal to
\begin{equation}\label{kappa1}
  k=\lambda^{2}_{1}+\lambda^{2}_{2}+\nu^{2}_{2},
\end{equation}
i.e. $(G', g, Q)$  is conformally flat.

Therefore, we establish the truthfulness of the following
\begin{prop}
 In Case $(A)$ the following properties of $(G', g, Q)$ are valid:
\begin{itemize}
\item[1)]  the components of $F$, $\theta$ and $\theta^{*}$ are \eqref{F-ex1} and \eqref{theta-ex1};
\item[2)]  the nonzero components of $R$ and $\rho$ are \eqref{R-ex1} and \eqref{rho-ex1};
 \item[3)]  the manifold is an Einstein manifold and the scalar curvature is \eqref{tau-ex1};
\item[4)] the manifold  is of constant sectional curvature \eqref{kappa1}.
\end{itemize}
\end{prop}

With similar calculations we obtain the following statements.

\begin{prop}
 In Case $(B)$ the Lie algebra is determined by the commutators \begin{equation*}
\begin{split}
  [x_{1}, x_{2}]&=\lambda_{1}x_{1}+\lambda_{2}x_{2}-(\lambda_{1}+\lambda_{2})x_{3},\\
    [x_{2}, x_{3}]&=-(\lambda_{1}+\lambda_{2})x_{1}+\lambda_{1}x_{2}+\lambda_{2}x_{3},\\
    [x_{1}, x_{3}]&=-\lambda_{2}x_{1}+(\lambda_{1}+\lambda_{2})x_{2}-\lambda_{1}x_{3}.\\
  \end{split}
\end{equation*}
In this case
\begin{itemize}
\item[1)] the components of $F$, $\theta$ and $\theta^{*}$ are
\begin{align*}\nonumber
F_{111} =F_{222}  = & F_{333} =2(\lambda_{1}-\lambda_{2}), \\
  F_{112} =F_{221}  = & F_{223}  =F_{313} =F_{332} = F_{113} =\frac{1}{2}(\lambda_{1}-\lambda_{2}), \\\nonumber
 F_{122}  = F_{123}= & F_{211} = F_{213} =F_{133}=F_{233} =\\\nonumber
  &=F_{311}= F_{312} =F_{322} =\lambda_{2}-\lambda_{1}, \\\nonumber
\theta_{1}=\theta_{2}=\theta_{3}=&3(\lambda_{1}-\lambda_{2}),\
\theta^{*}_{1}=\theta^{*}_{2}=\theta^{*}_{3}=\frac{3}{2}(\lambda_{2}-\lambda_{1});
\end{align*}
\item[2)]  the nonzero components of  $R$ and $\rho$ are
\begin{align*}
    R_{1212}= R_{2323}=R_{1313}=\frac{3}{4}(\lambda_{1}-\lambda_{2})^{2},\\
      \rho_{11}= \rho_{22}=\rho_{33}=-\frac{3}{2}(\lambda_{1}-\lambda_{2})^{2};
\end{align*}
 \item[3)]  the manifold is an Einstein manifold and the scalar curvature is \begin{equation*}
    \tau=-\frac{9}{2}(\lambda_{1}-\lambda_{2})^{2};
\end{equation*}
\item[4)] the manifold is of constant sectional curvature \begin{equation*}
  k=\frac{3}{4}(\lambda_{1}-\lambda_{2})^{2}.
\end{equation*}
\end{itemize}
\end{prop}

\begin{prop}
 In Case $(C)$ the Lie algebra is determined by the commutators  \begin{equation*}
\begin{split}
  [x_{1}, x_{2}]&=\lambda_{1}x_{1}+\lambda_{2}x_{2}-(\lambda_{1}+\lambda_{2})x_{3},\\
  [x_{2}, x_{3}]&=\lambda_{1}x_{1}+\lambda_{2}x_{2}-(\lambda_{1}+\lambda_{2})x_{3},\\
  [x_{1}, x_{3}]&=-\lambda_{1}x_{1}-\lambda_{2}x_{2}+(\lambda_{1}+\lambda_{2})x_{3}.\\
  \end{split}
\end{equation*}
In this case
\begin{itemize}
\item[1)] the components of $F$, $\theta$ and $\theta^{*}$ are zero;
\item[2)]  the nonzero components of $R$ and $\rho$ are
\begin{equation*}
      R_{1212}= R_{2323}=R_{1313}=R_{1332}=R_{1223}= R_{1321}=2(\lambda^{2}_{1}+\lambda_{1}\lambda_{2}+\lambda^{2}_{2}),
\end{equation*}
\begin{equation*}
        \rho_{11}= \rho_{22}=\rho_{33}= \rho_{12}=\rho_{13}= \rho_{23}=-4(\lambda^{2}_{1}+\lambda_{1}\lambda_{2}+\lambda^{2}_{2});
\end{equation*}
 \item[3)]  the manifold is an Einstein manifold and the scalar curvature is \begin{equation*}
    \tau=-12(\lambda^{2}_{1}+\lambda_{1}\lambda_{2}+\lambda^{2}_{2});
\end{equation*}
\item[4)] the manifold is of constant sectional curvature \begin{equation*}
  k=2(\lambda^{2}_{1}+\lambda_{1}\lambda_{2}+\lambda^{2}_{2});
\end{equation*}
\item[5)] it is valid $[Qx_{i}, Qx_{j}]=[x_{i}, x_{j}]$, i.e. $Q$ is an abelian structure.
\end{itemize}
\end{prop}

%---------------------------------------------------------

% ------------------------------------------------------------------------------------------------------------------------

% ----------------------------------------------------------------
\end{document}